\documentclass[12pt]{amsart}

\usepackage{amsthm,amsfonts,amsmath,amssymb,latexsym,epsfig}
\usepackage[usenames]{color}
\usepackage[all]{xy}

\DeclareMathAlphabet\oldmathcal{OMS}        {cmsy}{b}{n}
\SetMathAlphabet    \oldmathcal{normal}{OMS}{cmsy}{m}{n}
\DeclareMathAlphabet\oldmathbcal{OMS}       {cmsy}{b}{n} 

\usepackage{eucal}

\newtheorem{theorem}{Theorem}[section]
\newtheorem{lemma}[theorem]{Lemma}
\newtheorem{proposition}[theorem]{Proposition}
\newtheorem{corollary}[theorem]{Corollary}
\newtheorem{definition}[theorem]{Definition}
\newtheorem{question}[theorem]{Question}

\newenvironment{remark}{\medskip \refstepcounter{theorem}
\noindent  {\bf Remark \thetheorem}.\rm}{\,}

\newtheorem*{ack}{Acknowledgements}

\renewcommand{\thetheorem}{\thesection.\arabic{theorem}}

\def\<{\langle}

\def\>{\rangle}

\def\BOne{{\mathchoice {\rm 1\mskip-4mu l} {\rm 1\mskip-4mu l}
                          {\rm 1\mskip-4.5mu l} {\rm 1\mskip-5mu l}}}

\def\fract#1#2{\raise4pt\hbox{$ #1 \atop #2 $}}
\def\decdnar#1{\phantom{\hbox{$\scriptstyle{#1}$}}
\left\downarrow\vbox{\vskip15pt\hbox{$\scriptstyle{#1}$}}\right.}

\def\bbc{{\mathbb C}}

\def\bbp{{\mathbb P}}
\def\bbq{{\mathbb Q}}
\def\bbr{{\mathbb R}}

\def\bbt{{\mathbb T}}

\def\bbz{{\mathbb Z}}

\def\gra{\alpha}

\def\grd{\delta}
\def\gre{\epsilon}

\def\gro{\omega}

\def\grr{\rho}

\def\grz{\zeta}
\def\grD{\Delta}

\def\grG{\Gamma}

\def\grL{\Lambda}
\def\grO{\Omega}

\def\bfl{{\bf l}}

\def\bfF{{\bf F}}

\def\calo{{\mathcal O}}

\def\cald{{\mathcal D}}

\def\calf{{\mathcal F}}

\def\cali{{\mathcal I}}

\def\calm{{\mathcal M}}

\def\calo{{\mathcal O}}

\def\cals{{\oldmathcal S}}
\def\calS{{\mathcal S}}
\def\calt{{\mathcal T}}

\def\la#1{\hbox to #1pc{\leftarrowfill}}
\def\ra#1{\hbox to #1pc{\rightarrowfill}}

\def\ga{{\mathfrak a}}

\def\gt{{\mathfrak t}}

\def\gH{{\mathfrak H}}

\def\gN{{\mathfrak N}}

\def\lra{{\longrightarrow}}

\def\Hol{\rm Hol}

\def\hook{\mathbin{\hbox to 6pt{%
                 \vrule height0.4pt width5pt depth0pt
                 \kern-.4pt
                 \vrule height6pt width0.4pt depth0pt\hss}}}

\title{Transverse K\"ahler holonomy in Sasaki Geometry and $\cals$-Stability} 
\author{Charles P. Boyer}
\author{Hongnian Huang}
\author{Christina W. T{\o}nnesen-Friedman}

\thanks{The first author was partially supported by grant \#519432 from the Simons Foundation. The third author was partially supported by grant \#422410 from the Simons Foundation.}

\date{\today}
\address{Charles P. Boyer, Department of Mathematics and Statistics,
University of New Mexico, Albuquerque, New Mexico 87131, USA.}
\email{cboyer@unm.edu} 
 \address{Hongnian Huang, Department of Mathematics and Statistics,
 University of New Mexico, Albuquerque, New Mexico 87131, USA.}
 \email{hnhuang@unm.edu}
\address{Christina W. T{\o}nnesen-Friedman, Department of Mathematics, Union
College, Schenectady, New York 12308, USA } \email{tonnesec@union.edu}

\begin{document}

\begin{abstract}
We study the transverse K\"ahler holonomy groups on Sasaki manifolds $(M,\cals)$ and their stability properties under transverse holomorphic deformations of the characteristic foliation by the Reeb vector field. In particular, we prove that when the first Betti number $b_1(M)$ and the basic Hodge number $h^{0,2}_B(\cals)$ vanish, then $\cals$ is stable under deformations of the transverse K\"ahler flow. In addition we show that an irreducible transverse hyperk\"ahler Sasakian structure is $\cals$-unstable, whereas, an irreducible transverse Calabi-Yau Sasakian structure is $\cals$-stable when $\dim M\geq 7$. Finally, we prove that the standard Sasaki join operation (transverse holonomy $U(n_1)\times U(n_2)$) as well as the fiber join operation preserve $\cals$-stability.

\end{abstract}

\maketitle

\markboth{Sasaki-Transverse K\"ahler}{C. P. Boyer, H. Huang, C. W. T{\o}nnesen-Friedman}

%\tableofcontents

\section{Introduction}
It is well known from Berger's classification of Riemannian holonomy that the irreducible holonomy groups in K\"ahler geometry are precisely, $U(n),SU(n)$ and $Sp(n)$ which correspond to irreducible K\"ahler, Calabi-Yau, and hyperk\"ahler geometry, respectively. There is also a well known Stability Theorem of Kodaira and Spencer \cite{KoSp60} that says that any infinitesimal deformation of a compact complex manifold which is K\"ahler remains K\"ahler. A similar result was obtained in the other two cases by Goto \cite{Got04}. Analogues of these stability theorems for holomorphic foliations was proven by El Kacimi Alaoui and Gmira in \cite{ElKGm97} in the K\"ahler case, and by Moriyama \cite{Mor10}  in the Calabi-Yau case. See also \cite{TomVe08}. The two special holonomy cases have been studied further by Habib and Vezzoni \cite{HaVe15}. In particular, they prove that a transverse K\"ahler foliation admits a transverse hyperk\"ahler structure if and only if it admits a transverse hyperhermitian structure. This is the transverse version of a result of Verbitsky \cite{Ver05} in the compact K\"ahler manifold case.

The purpose of this paper is to study these transverse versions and their relationship to Sasaki geometry. It is well known \cite{KoSp60} that there are obstructions, namely the Hodge numbers $h^{0,2}$, for deformations of projective algebraic structures to remain projective algebraic. The point is that the transverse K\"ahler structure of a Sasakian structure is algebraic in an appropriate sense, cf. Section 7.5 of \cite{BG05}. A Sasakian structure $\cals$ is said to be $\cals$-stable (or $\cals$-rigid) if every sufficiently small transverse K\"ahlerian deformation of $\cals$ remains Sasakian. So an important question is 

\begin{question}\label{stabquest}
Which Sasakian structures are $\cals$-stable and which are $\cals$-unstable?
\end{question}

It was recently shown by Nozawa \cite{Noz14} that Sasaki nilmanifolds of dimension at least 5 are $\cals$-unstable, that is, their transverse K\"ahler deformations become non-algebraic. This is done by deforming the transverse K\"ahler flow on a Sasaki nilmanifold, i.e on the total space of an $S^1$ bundle over an Abelian variety of complex dimension at least two. These nilmanifolds are discussed briefly in Section \ref{trivialhol}. Building on results of Nozawa \cite{Noz14} we obtain the first main result of this paper:

\begin{theorem}\label{no02thm}
Let $(M,\cals)$ be a Sasaki manifold with vanishing first Betti number and such that the basic Hodge numbers satisfy $h^{0,2}_B=h^{2,0}_B=0$. Then $\cals$ is $\cals$-stable.
\end{theorem}

Moriyama's Stability Theorem \cite{Mor10} for irreducible transverse Calabi-Yau structures follows as a special case of Theorem \ref{no02thm}.

\begin{corollary}\label{CYcor}
Let $(M,\cals)$ be a Sasaki manifold of dimension $2n+1$ with $n>2$ and transverse holonomy group equal to $SU(n)$ Then $(M,\cals)$ is $\cals$-stable. Moreover, the local universal deformation space is isomorphic to an open set in $H^1(M,\Theta)$.
\end{corollary}

%Recent work of Goertsches, Nozawa, and T\"oben \cite{GoNoTo12} shows that the basic Hodge numbers only depend on the isomorphism class of Sasaki CR structures. positive Sasakian structures and toric Sasakian structures are $\cals$-stable. 

Here $\Theta$ is the sheaf of transversally holomorphic vector fields on $M$. For irreducible transverse hyperk\"ahler geometry the contrary holds which gives the second main result of the paper.

\begin{theorem}\label{transhyperkahthm}
Let $(M,\cals)$ be a Sasaki manifold with vanishing first Betti number and a compatible irreducible transverse hyperk\"ahler structure. Then $(M,\cals)$ is $\cals$-unstable.
Moreover, the local universal deformation space is isomorphic to an open set in $H^1(M,\Theta)$.
\end{theorem}

Much more can be said about irreducible transverse hyperk\"ahler structures in dimension 5. First, a classification of simply connected 5-manifolds that admit null Sasakian structures has just recently been completed \cite{CMST20} by proving the existence of orbifold K3 surfaces $X$ with second Betti number $b_2(X)=3$. This completes the classification initiated in \cite{BGM06,BG05,Cua14}. A simply connected 5-manifold which admits a null Sasakian structure is diffeomorphic to a k-fold connected sum
\begin{equation}\label{nullsas5man}
\# k(S^2\times S^3) \quad \text{with $k=2,\ldots,21$}
\end{equation}
and each such 5-manifold admits a null Sasakian structure. These are represented as $S^1$ orbibundles over $K3$ orbifolds $X_k$ with $b_2(X_k)=k+1$ and $\pi_1^{orb}(X_k)=\BOne$. A smooth K3 surface is diffeomorphic to $X_{22}$. Furthermore, any 5-manifold of the form $\# k(S^2\times S^3)$ admits positive Sasakian structures (cf. Corollary 11.4.8 of \cite{BG05})  which are stable by \cite{Noz14}, so Theorems \ref{no02thm} and \ref{transhyperkahthm} give

\begin{corollary}\label{mainthm}
Any null Sasakian structure on a simply connected 5-manifold $M$ is $\cals$-unstable, and all such $M$ are of the form of Equation \eqref{nullsas5man}. So the manifolds $\# k(S^2\times S^3)$ with $k=2,\ldots,21$ admit both $\cals$-stable and $\cals$-unstable Sasakian structures.
\end{corollary}

Given this corollary and Nozawa's result for Sasaki nilmanifolds one might wonder whether every null Sasakian structure is $\cals$-unstable. But this is not true for non-trivial $S^1$ bundles over an Enriques surfaces $E$ even though these are smooth $\bbz_2$ quotients of $X_{22}$. Since Enriques surfaces are projective and have $h^{0,2}(E)=h^{2,0}(E)=0$, we have the following corollary of Theorem \ref{no02thm}:

\begin{corollary}\label{nullSasstab}
Let $(M^5,\cals)$ be a regular Sasakian structure over an Enriques surface. Then $(M^5,\cals)$ is $\cals$-stable.
\end{corollary}

\begin{remark} 
One can consider Enriques surfaces as K3 orbifolds with a trivial orbifold structure. Its canonical bundle $K_E$ is not trivial, but $K^2_E$ is. On the other hand K3 orbifolds of the form $X_{22}/G$ where $G$ is a finite group acting on $X_{22}$ that leaves its holomorphic $(2,0)$ form invariant have $\pi_1^{orb}(X_{22}/G)=G$ and a trivial canonical bundle. They have been studied \cite{Nik76,Fuj83} and classified by Mukai \cite{Muk88}. As pointed out by Koll\'ar \cite{Kol05b} the best known example where $\pi_1^{orb}(X)\neq \BOne$ is the well known Kummer surface $X=\bbt^2/\bbz_2$ in which case $\pi_1^{orb}(X)$ is an extension of $\bbz_2$ by $\bbz^4$. It would be interesting to determine the stability properties of these structures.
\end{remark}

%A K3 orbifold can be characterized by the existence of its complex symplectic $(2,0)$-form.  there are orbifold K3 surfaces with $\pi_1^{orb}\neq \BOne$, for example the well known Kummer surface. 

%We refer to Chapter 4 of \cite{BG05} for the fundamentals as well as standard references \cite{ALR07,MoMr03,Moe02}. 

\begin{ack}
The authors thank Georges Habib for pointing out an incorrect lemma in an earlier version of this paper (see Remark \ref{transKahcohorrem} below), and for his interest in our work. We also thank an anonymous  referee for suggesting improvements in the exposition.
\end{ack}

\section{The Transverse K\"ahler Flow and Sasakian Structures}
An oriented 1-dimensional foliation is called a {\it flow}, and we are interested in transverse Hermitian and transverse K\"ahler flows. In particular they are {\it Riemannian flows}, and they form a special case of transverse K\"ahler foliations which were studied by El Kacimi Alaoui and collaborators \cite{ElK,ElKGm97}. Their relation with Sasaki geometry was developed in \cite{BG05} and developed further by Nozawa and collaborators \cite{Noz14,GoNoTo12}. We begin with the Riemannian foliations $\calf$ (see \cite{Mol88}, chapter 2 of \cite{BG05}, and references therein) on a compact oriented manifold $M$ and its basic cohomology ring $H^*_B(\calf)$. A Riemannian foliation $\calf$ is said to be {\it homologically oriented} if $H^{n}_B(\calf)\neq 0$ where $n$ is the (real) codimension of $\calf$. If the foliation $\calf$ is holomorphic with transverse complex structure $\bar{J}$ and has a compatible transverse Riemannian metric $g^T$ such that $g^T\circ \bar{J}\otimes\BOne=\gro^T$ is a basic 2-form, the triple $(\calf,\bar{J},\gro^T)$ is called a {\bf transverse Hermitian foliation}.

Note that for a Riemannian flow $\calf$ a choice of Riemannian metric on $M$ of the form $g=g^T+\eta\otimes\eta$, where $\eta$ is the dual 1-form to a nowhere vanishing section $\xi$ of $\calf$, splits the exact sequence
$$0\ra{2.5} \calf \ra{2.5} TM \ra{2.5} TM/\calf \ra{2.5} 0$$
as 
\begin{equation}\label{orthogsplitM}
TM=\calf\oplus \cald,
\end{equation}
and we have identified $\gro^T$ with a 2-form on $\cald$, also denoted $\gro^T$. The splitting gives an isomorphism of complex vector bundles $(\cald,J)\approx (TM/\calf,\bar{J})$. If the transverse  flow $\calf$ is also homologically orientable, it follows from Molino and Sergiescu \cite{MoSe85} that there exists a Riemannian metric $g$ and a nowhere vanishing vector field $\xi$ tangent to $\calf$ such that $\xi$ is a Killing vector field with respect to $g$. In this case the flow $\calf$ is said to be {\bf isometric} and the pair $(g,\xi)$ is called a {\bf Killing pair} in \cite{Noz14}. 
Furthermore, the orbits generated by the Killing field $\xi$ are geodesics of $g$. Conversely an isometric Riemannian flow is homologically orientable.

Without loss of generality we can take $\xi$ to be a unit vector field, in which case we see that its dual {\it characteristic} 1-form $\eta$ satisfies 
$$\eta(\xi)=1, \qquad \xi\hook d\eta=0.$$
This implies that $d\eta$ is basic and its basic cohomology class $[d\eta]_B$, called the {\bf basic Euler class} of the isometric flow $\calf$ by Saralegui \cite{Sar85}, is, up to multiplication, an invariant of the foliation. Note that $d\eta$ depends on the metric $g$, but its vanishing does not.  A transverse Hermitian flow is said to be {\bf trivial} if its basic Euler class vanishes in which case $(M,\calf)$ is a foliated bundle \cite{Sar85}. For example $[d\eta]_B=0$ if a finite cover of $M$ is diffeomorphic to the product $N\times S^1$ with the foliation by circles in the $S^1$ direction. We deal almost exclusively with nontrivial transverse Hermitian flows, that is we assume that $M$ admits a one dimensional flow $\calf$ with a non-zero Euler class, i.e. $[d\eta]_B\neq 0$. Note that if $b_1(M)=0$, every transverse Hermitian flow on $M$ is nontrivial. We only consider nontrivial isometric transverse Hermitian flows which we write as the quadruple $(\xi,\eta,\Phi,g)$ or $(\xi,\eta,\Phi,\gro^T)$ depending on the emphasis where 
\begin{equation}\label{riemmet}
g= \gro^T\circ (\BOne \otimes J) +\eta\otimes\eta
\end{equation}
where the endomorphism $\Phi$ is
\begin{equation}\label{Phieqn}
\Phi = J \oplus  (\BOne -\xi\otimes\eta) 
\end{equation}
and $J$ is the complex structure on $\cald$. The following equations hold:
\begin{equation}\label{Phieqn2}
\pounds_\xi\gro^T=0, \quad \pounds_\xi\Phi=0,
 \quad \Phi^2= \begin{cases} 
 J^2=-\BOne  \quad   &\text{on $\cald$,} \\
                0 &\text{on $\calf$.} 
              \end{cases}
\end{equation}
We emphasize here that the quadruple $(\xi,\eta,\Phi,g)$ is not necessarily a contact metric structure since $\gro^T$ is not necessarily $d\eta$. We also see 

\begin{lemma}\label{CRlem}
The pair $(\cald,J)$ defines a CR structure on $M$
\end{lemma}

\subsection{Transverse K\"ahler flows}
It is well known \cite{BG05,Noz14} that the characteristic Reeb foliation $\calf_\xi$ of a Sasakian structure $\cals=(\xi,\eta,\Phi,g)$ is an isometric transverse K\"ahler flow with a nontrivial Euler class, i.e. $[d\eta]_B\neq 0$. However, the converse does not generally hold and one of the goals of this paper is to describe their relationship. 

\begin{definition}\label{transkahflow}
A transverse Hermitian flow $(\calf,\bar{J},\gro^T)$ is said to be a {\bf transverse K\"ahler flow} if the basic 2-form $\gro^T$ is closed.
\end{definition}

\begin{remark}\label{transKahcohorrem}
As was pointed out to us by Georges Habib not every transverse K\"ahler flow is homologically orientable; a counterexample is given by Carri\`ere. See Example 6 on page 39 in \cite{Car84b}, and also Appendix A in \cite{Mol88}. Hence, not every transverse K\"ahler flow can be taken to be isometric. However, henceforth in this paper, all transverse K\"ahler structures are homologically oriented, and so can be taken to be isometric.
\end{remark}

%From Lemma \ref{transKahcohor} and the preceeding discussion, we can represent a transverse K\"ahler flow by the quadruple $(\xi,\eta,\Phi,\gro^T)$ or $(\xi,\eta,\Phi,g)$. However, such a quadruple is not unique as shown in Lemma \ref{EulerKah} below. 

As emphasized by El Kacimi-Alaoui \cite{ElK}, homologically oriented transverse K\"ahler foliations on compact manifolds possess the same properties as K\"ahler structures on compact manifolds, the Hodge decomposition, Lefschetz decomposition, etc. In particular, a homologically oriented transverse K\"ahler flow $(\xi,\eta,\Phi,g)$ gives rise to a basic Hodge decomposition (over $\bbc$),
\begin{equation}\label{Hodgedecomp}
H^{n}_B(\calf)=\bigoplus_{p+q=n}H^{p,q}_B(\calf), \qquad H^{q,p}_B(\calf)=\overline{H^{p,q}_B(\calf)},
\end{equation}
which gives the basic Hodge numbers
\begin{equation}\label{basicHodgenumb}
h^{p,q}_B=\dim_\bbc H^{p,q}_B(\calf).
\end{equation}
We let $b_B^1=\dim H^1_B(M,\bbr)$ denote the basic first Betti number. Then we easily see

\begin{lemma}\label{b10lem}
Let $(M,\calf)$ be a compact manifold $M$ with an isometric transverse K\"ahler flow $\calf$. Then the natural map $H^1_B(\calf)\lra H^1(M,\bbc)$ is injective,
and $b_B^1=2h^{1,0}_B=2h^{0,1}_B$. In particular, $b_1(M)=0$ implies $h^{1,0}_B=h^{0,1}_B=0$.
\end{lemma}

%\begin{definition}\label{isotranskahflow}
%An isometric transverse Hermitian flow $(\xi,\eta,\Phi,g)$ is said to be a {\bf transverse K\"ahler flow} if the basic 2-form $\gro^T$ is closed.
%\end{definition}

%Given a holomorphic foliation $(\calf,\bar{J})$ we have 

\begin{definition}\label{transkahcone}
We define the {\bf transverse K\"ahler cone} $K^T(M,\calf)$ to be the set of all transverse K\"ahler classes $[\gro^T]_B$ in $H^{1,1}_B(\calf)\cap H^2_B(\calf)$. 
\end{definition}

From the standard definition of Sasakian structure one sees

\begin{lemma}\label{trKahsaslem}
A transverse K\"ahler flow $(\calf,\bar{J},\gro^T)$ is Sasakian if and only if $\gro^T=d\eta$ where $\eta$ is a contact 1-form and it is isometric with respect to the Reeb vector field of $\eta$.
\end{lemma}

\begin{proof}
The only if part is well known. Suppose that $\gro^T=d\eta$ which implies that $d\eta$ is type $(1,1)$. The transverse form $(\gro^T)_B^n$ is a basic volume form, so $\eta\wedge (\gro^T)^n$ is nowhere vanishing. Thus, $\eta$ is a contact 1-form on $M$. So defining $\xi$ to be the Reeb vector field of $\eta$, the quadruple $(\xi,\eta,\Phi,g)$ is a contact metric structure where Equations \eqref{riemmet} and \eqref{Phieqn} hold. Furthermore, since $(\calf,\bar{J},\gro^T)$ is isometric with respect to $\xi$, \eqref{Phieqn2} holds which implies that $(\xi,\eta,\Phi,g)$ is Sasakian.  
\end{proof}

\subsection{Transverse Holonomy}
For Riemannian foliations we consider the holonomy group of the transverse Levi-Civita connection. 

\begin{definition}\label{transholgrp}
For Riemannian foliations $\calf$ the {\bf transverse holonomy group} $\Hol(\calf)$ is the Riemannian holonomy group of the transverse Levi-Civita connection $\nabla^T$, and $\Hol^0(\calf)$ denotes the restricted transverse holonomy group\footnote{Recall that the restricted holonomy group is obtained by restricting the holonomy computation to null-homotopic loops. $\Hol^0(\calf)$ is the connected component of $\Hol(\calf)$.}. 
\end{definition}

In the case of transverse K\"ahler structures, the transverse complex structure is also parallel, i.e. $\nabla^T \bar{J}=0$ and equivalently, $\nabla^TJ=0$ using the isomorphism defined by the splitting \eqref{orthogsplitM}. It follows that 

%For our Riemannian foliation $\calf_\xi$ we have the tranverse Levi-Civita connection $\nabla^T$ on the vector bundle $\cald$, and since the transverse structure is K\"ahler, 

\begin{lemma}\label{transhollem}
Let $M$ be a compact manifold of dimension $2n+1$ with a transverse Hermitian flow $(\calf_\xi,\bar{J},\gro^T)$. Then the holonomy representation ${\rm Hol}(\calf_\xi,\bar{J},\gro^T)$ on $TM/\calf_\xi$ lies in $U(n)\subset GL(n,\bbc)$  if and only if $(\calf_\xi,\bar{J},\gro^T)$ is K\"ahler. Moreover, isomorphism 
$$(TM/\calf_\xi,\bar{J},\gro^T)\fract{\psi}{\ra{3.5}} (\cald,J,\gro^T)$$ 
induced by the splitting \eqref{orthogsplitM} induces an isomorphism of holonomy representations.
\end{lemma}

\begin{proof}
Since the holonomy representation is defined only up to conjugation in $GL(n,\bbc)$, the isomorphism $\psi$ implies that  holonomy representions of $\Hol(\calf)$ on $TM/\calf_\xi$ and $\cald$ are represented by conjugate subgroups of $GL(n,\bbc)$. Moreover, $\Hol(\calf)\subset U(n)$ if and only if $\nabla^T\gro^T=0,~\nabla^T\bar{J}=0$ if and only if $\gro^T$ is a basic closed 2-form if and only if the transverse Hermitian structure $(\calf,\bar{J},\gro^T)$ is K\"ahler.
\end{proof}

Equivalently, we state this with respect to the splitting \eqref{orthogsplitM}.

\begin{lemma}\label{Kahhol}
Let $M$ be a compact manifold with an isometric transverse Hermitian flow $(\xi,\eta,\Phi,g)$. Then $\Hol(\calf)\subset U(n)$ if and only if $(\xi,\eta,\Phi,g)$ is K\"ahler.
\end{lemma}

The irreducible holonomy groups that are proper subgroups of $U(n)$ are $SU(n)$ and $Sp(\frac{n}{2})$ where the later occurs only for $n$ even. In this paper we are interested in transverse K\"ahler flows $(\calf,\bar{J},\gro^T)$ whose transverse holonomy groups are either $SU(n)$  (transverse Calabi-Yau) or $Sp(n)$ (transverse hyperk\"ahler).

\subsection{The Invariant Torus and its Invariant Cone}
For a homologically oriented Riemannian flow on a compact manifold there is a torus of isometries as described by the work of Molino and his coworkers \cite{Mol79,Mol82,MoSe85,Mol88,Car84b} which we now describe. A result of Carri\`ere \cite{Car84b} (see also Appendix A of \cite{Mol88}) states that the closure $\overline{\calf}$ of a Riemannian flow $\calf$ on a compact manifold is a singular Riemannian foliation whose leaf closures are diffeomorphic to a real torus $\bbt$ and that $\calf$ restricted to a leaf is conjugate to a linear flow on $\bbt$. In the case of an isometric transverse K\"ahler flow of real dimension $2n+1$ we have the range $1\leq k\leq n+1$ for the dimension $k$ of $\bbt$. The dimension $k$ is an invariant of the flow called its {\it toral rank}. Hence, associated to each quadruple $(\xi,\eta,\Phi,\gro^T)$ is a maximal torus $\bbt^k$ that leaves $(\xi,\eta,\Phi,\gro^T)$ invariant. $\bbt^k$ is called the {\bf invariant torus} of the isometric transverse K\"ahler flow $(\xi,\eta,\Phi,\gro^T)$. This gives rise to the invariant cone in \eqref{transkahsascone} below, keeping in mind that if the transverse K\"ahler flow is Sasakian it coincides with the Sasaki cone, and is thus viewed as a generalization of the Sasaki cone \cite{BGS06}.

Applying the basic Hodge decomposition \eqref{Hodgedecomp} to the basic Euler class gives 
\begin{equation}\label{basEuldecomp}
[d\eta]_B=[d\eta^{2,0}]_B +[d\eta^{1,1}]_B +[d\eta^{0,2}]_B, \qquad d\eta^{0,2} =\overline{d\eta^{2,0}}
\end{equation}
and as we shall see below $d\eta^{2,0}$ (equivalently $d\eta^{0,2}$) is an obstruction for the transverse K\"ahler flow to be Sasakian. We have 

\begin{proposition}\label{Abprop}
Let $(\xi,\eta,\Phi,g)$ be an isometric transverse K\"ahler flow on a compact manifold $M$ of dimension $2n+1$. Then there exists a $k$-dimensional  Abelian Lie algebra $\ga(M,\calf)$ of Killing vector fields with $1\leq k\leq n+1$ that is independent of the transverse K\"ahler metric and commutes with all transverse vector fields. Moreover, $\xi\in \ga(M,\calf)$.
\end{proposition}

\begin{proof}
This follows from a result of Molino (see Theorem 5.2 in \cite{Mol88}) that says that on any compact manifold with a Riemannian foliation there exists a locally constant sheaf of germs $C(M,\calf)$ of locally transverse commuting Killing fields such that 
\begin{enumerate}
\item all global transverse vector fields commute with $C(M,\calf)$,
\item  $C(M,\calf)$ is independent of the transverse metric $g^T$.
\end{enumerate}
$C(M,\calf)$ is called the {\it commuting sheaf} in \cite{Mol88} and the {\it faisceau transverse central} in \cite{Mol79,Mol82}, and it is an invariant of the foliation $\calf$. We apply this to the case that $M$ has a transverse K\"ahler flow $(\calf,\bar{J},\gro^T)$. In this case $\calf$ is homologically oriented, so by \cite{MoSe85} the sheaf $C(M,\calf)$ has a global trivialization. This gives an Abelian Lie algebra $\ga^T(M,\calf)$ of global transverse vector fields associated to $\calf$ that is independent of the transverse metric and commutes with all transverse vector fields. Thus, there is a nowhere vanishing smooth vector field $\xi$ tangent to $\calf$ that commutes with all transverse vector fields. This implies that $\ga^T(M,\calf)$ extends to a $k$-dimensional Abelian Lie algebra
\begin{equation}\label{ablie}
\ga(M,\calf_\xi) = \ga^T(M,\calf) \oplus \bbr\xi
\end{equation}
with $1\leq k\leq n+1$ which is independent of the transverse K\"ahler metric and commutes with all transverse vector fields.
Moreover, since the elements of $\ga^T(M,\calf)$ are Killing fields with respect to any transverse K\"ahler metric and transverse K\"ahler forms are harmonic with respect to the basic Laplacian $\grD_B$, the elements of $\ga(M,\calf)$ also leave $\bar{J}$ invariant. Clearly, by construction $\xi\in \ga(M,\calf_\xi)$.
\end{proof}

We now define the transverse K\"ahler analogue of the Sasaki cone, namely the {\bf invariant cone}
\begin{equation}\label{transkahsascone}
\ga^+(M,\calf_\xi)= \{ \xi'\in \ga(M,\calf_\xi)~|~ \eta(\xi')>0\}
\end{equation}
where $\eta$ is the dual 1-form to $\xi$. Such elements take the form 
\begin{equation}\label{Reebsascone}
\xi'=\bar{X} +\eta(\xi')\xi 
\end{equation}
where $\bar{X}\in \ga^T(M,\calf_\xi)$. The function $\eta(\xi')$ is basic with respect to $\calf_\xi$ and is called a {\bf Killing potential}. 

\begin{lemma}\label{isocone}
The set $\ga^+(M,\calf_\xi)$ is a convex cone. Moreover, for any $\xi'\in \ga^+(M,\calf_\xi)$ we have $\ga^+(M,\calf_{\xi'})=\ga^+(M,\calf_\xi)$.
\end{lemma}

\begin{proof}
Consider the transverse homothety defined by $\xi\mapsto a^{-1}\xi,~\eta\mapsto a\eta$ and $\gro^T\mapsto a\gro^T$. Then from \eqref{riemmet} we see that if $\xi$ is in $\ga^+(M,\calf_\xi)$, then so is  $(g_a,a^{-1}\xi)$ where 
$$g_a= a\gro^T\circ (\BOne\oplus \bar{J})+ a^2\eta\otimes \eta.$$
So $\ga^+(M,\calf_\xi)$ is a cone. Now suppose $\xi_0,\xi_1\in \ga^+(M,\calf_\xi)$, and consider the line segment $\xi_t=\xi_0 +t(\xi_1-\xi_0)$ then 
$$\eta_0(\xi_t)= (1-t)+t\eta_0(\xi_1)>0$$
for all $t\in [0,1]$ implying that $\ga^+(M,\calf_\xi)$ is convex. To prove the last statement we note that if $\xi'\in\ga(M,\calf_\xi)$, than $\ga(M,\calf_\xi)$ and $\ga(M,\calf_\xi')$ are isomorphic Abelian Lie algebras. So we only need to prove that $\xi\in\ga^+(M,\calf_{\xi'})$. But this follows by construction since $\eta'(\xi)=(\eta(\xi'))^{-1}>0$.
\end{proof}

Next as in Lemma 2 of \cite{ApCa18} we have:

\begin{lemma}\label{AClem}
Suppose $\xi_0,\xi_1\in \ga^+(M,\calf_\xi)$ then $\xi_1=\eta_0(\xi_1)\xi_0\mod\ker\eta_0$.
\end{lemma}

\begin{remark}\label{choosesection}
In the Sasaki category fixing a Sasakian structure $\cals_0=(\xi_0,\eta_0,\Phi_0,g_0)$ and a nowhere vanishing smooth function $f$ the vector field $f\xi_0$ defines a {\it weighted Sasakian structure} in the sense of \cite{ApCa18} when $f$ is chosen to be a nowhere vanishing Killing potential $\eta_0(\xi_1)$ with respect to $\cals_0$. 
\end{remark}

%Fixing the basic Euler class does not uniquely determine the Killing pair, but fixing the contact bundle $\cald$ fixes the quadruple $(\xi,\eta,J,\gro^T)$ up to transverse homothety. Clearly, in any case $\bbt$ is contained in a maximal torus $\bbt_{max}$. 

\begin{proposition}\label{invtorprop}
The invariant torus $\bbt^k$ is independent of the transverse K\"ahler metric $\gro^T$ and the choice of Reeb field in $\ga^+(M,\calf)$. Hence, it is independent of the pair $(\xi,[\gro^T]_B)\in \ga^+(M,\calf)\times K^T(M,\calf)$. 
\end{proposition}

If we fix an isometric transverse K\"ahler flow $(\xi_o,\eta_o,\Phi_o,g_o)$ we obtain a family of isometric transverse K\"ahler flows associated to $(\xi_o,\eta_o,\Phi_o,g_o)$, namely the disjoint union
$$\bigsqcup_{\xi\in \ga^+(M,\calf_o)}K^T(M,\calf_\xi)=\{(\xi,\eta,\Phi,\gro^T) ~|~\xi\in \ga^+(M,\calf_{\xi_o}),~ [\gro^T]_B\in K^T(M,\calf_{\xi_o})\}$$
which is isomorphic to the diagonal in the product $\ga^+(M,\calf_{\xi_1})\times K^T(M,\calf_{\xi_2})$.

%We have a family of transverse K\"ahler flows associated to the transverse holomorphic flow $(\calf,\bar{J})$
%$$\calS(\calf,\bar{J})=\{(\calf_{\xi'},\bar{J},\gro^T) ~|~\xi'\in \ga^+(M,\calf),~ [\gro^T]_B\in K^T(M,\calf_{\xi'})\}.$$

As in the Sasaki case, Section 7.5.1 of \cite{BG05}, we give this family the $C^\infty$ compact-open topology as sections of vector bundles. This gives a smooth family of transverse K\"ahler flows within a fixed basic cohomology class $[\gro^T]_B$, and as in Section 6 of \cite{BGS06} we obtain a smooth family when fixing the underlying CR structure and varying $\xi'\in \ga^+(M,\calf)$ which implies that the family $\calS(\calf,\bar{J})$ is smooth.

\begin{lemma}\label{EulerKah}
Let $(g,\xi)$ and $(g',\xi)$ be two Killing pairs associated to the transverse K\"ahler flow $(\calf_\xi,\bar{J},\gro^T)$ with the same basic Euler class. Then there exists a basic 1-form $\grz$ such that 
$$g'=g +\grz\otimes\eta \oplus \eta\otimes \grz \oplus \grz\otimes\grz.$$
\end{lemma}

\begin{proof}
The dual 1-forms $\eta,\eta'$ satisfy $g(\xi,X)=\eta(X)$ and $g'(\xi,X)=\eta'(X)$, and since $[d\eta']_B=[d\eta]_B$ there exists a basic 1-form $\grz$ such that $d\eta'=d\eta +d\grz$. But $g$ is given by Equation \eqref{riemmet} and 
$$g'=\gro^T\circ (\BOne \otimes J) \oplus \eta'\otimes\eta'$$
which gives the result.
\end{proof}

\begin{remark}\label{fixCR}
Note that fixing the CR structure fixes the Killing pair $(g,\xi)$.
\end{remark}

\begin{remark}\label{gaurem}
The contact 1-form $\eta$ in a quasiregular Sasakian structure can be viewed as a connection in a principal $S^1$ orbibundle over a projective algebraic orbifold. Two such connections forms $\eta,\eta'$ are said to be {\bf gauge equivalent} if there exists a smooth basic function $f$ such that $\eta'=\eta +df$. One easily sees that such gauge transformed contact metric structures of a Sasakian structure are all Sasakian. This gives rise to gauge equivalences classes of Sasakian structure. Moreover, gauge equivalent Sasakian structures have the same underlying transverse K\"ahler flow, and a choice of Killing pair uniquely determines the gauge together with the class $[\grz]_B$ in $H^1_B(\calf)_\bbr\approx H^1(M,\bbr)$.
\end{remark}

%As in the Sasaki case we define a Riemannian metric $g$ on $M$ by
%where $\xi$ is a smooth nowhere vanishing section of $\calf$ and $\eta$ is its dual 1-form making \eqref{orthogsplit} an orthogonal splitting. Clearly $\xi$ satisfies $\xi\hook \gro^T=0$ since $\gro^T$ is basic. We fix an orientation for $\calf$ which since $M$ is oriented fixes the orientation of $\cald$. The fact that our transverse K\"ahler flows have nonvanishing basic Euler class implies that the subbundle $\cald$ is not integrable.

\section{Deformation Theory of Transverse K\"ahler Flows}
We begin by discussing the deformation theory of transverse holomorphic foliations.
The well known Kodaira-Spencer deformation theory of complex manifolds has been completed by Kuranishi \cite{Kur71} and applied to other pseudogroup structures \cite{Kod60,KoSp61}. In particular the deformation theory of transverse holomorphic foliations has been studied extensively \cite{DuKa79,DuKa80,G-M,GHS}. See also Section 8.2.1 of \cite{BG05}.  Since the characteristic foliation $\calf_\xi$ of a Sasakian structure is a transverse holomorphic foliation of dimension one, we can apply this theory to $\calf_\xi$ when the manifold is compact. We can parameterize the transverse complex structures on a Sasaki manifold $M$ by a complex analytic scheme\footnote{Schemes are needed here since the map $\Psi$ in Theorem \ref{Kurspace} vanishes to first order.} $(S,0)^T$ that is the zero set of a finite number of holomorphic functions and a (not necessarily reduced) germ at $0$. The main result is the following theorem of Girbau, Haefliger, and Sundararaman 

\begin{theorem}[\cite{GHS}]\label{Kurspace}
Let $\calf$ be a transverse holomorphic foliation on a compact
manifold $M,$ and let $\Theta_\calf$ denote the sheaf of germs of
transversely holomorphic vector fields. Then
\begin{enumerate}
\item There is a germ $(S,0)^T$ of an analytic space (called the Kuranishi space) parameterizing
a germ of a deformation $\calf_s$ of $\calf$ such that if
$\calf_{s'}$ is any germ of a deformation parameterizing $\calf$
by the germ $(S',0)^T$, there is a holomorphic map
$\phi:(S',0)^T\ra{1.3} (S,0)^T$ so that the deformation
$\calf_{\phi(s)}$ is isomorphic to $\calf_{s'}.$ 
\item The Kodaira-Spencer map $\grr:T_0S\ra{1.3} H^1(M,\Theta_\calf)$ is an
isomorphism. 
\item There is an open neighborhood $U\subset
H^1(M,\Theta_\calf)$ and a holomorphic map $\Psi:U\ra{1.3}
H^2(M,\Theta_\calf)$ such that $(S,0)^T$ is the germ at $0$ of
$\Psi^{-1}(0).$ The $2$-jet of $\Psi$ satisfies
$j^2\Psi(u)=\frac{1}{2}[u,u].$
\end{enumerate}
\end{theorem}

Here $u$ is a 1-form with coefficients in the sheaf $\Theta_\calf$, and the element $[u,u]\in H^2(M,\Theta_\calf)$ is the primary obstruction to performing the deformation. Item (i) of Theorem
\ref{Kurspace} says that the analytic space $(S,0)^T$ is versal. Moreover, if $H^2(M,\Theta_\calf)=0$ then a versal deformation exists and the Kuranishi space $(S,0)^T$ is isomorphic to an neighborhood of $0$ in $H^1(M,\Theta_\calf)$. Note that as described in \cite{GHS} Kodaira-Spencer-Kuranishi deformation theory works equally well on compact complex orbifolds. 

Given this first order isomorphism, it is natural to ponder whether actual deformations exist and what their set of equivalences are, that is, describe the moduli space. However, here we restrict ourselves to paint a picture of the local moduli space, namely, the Kuranishi space of deformations of transverse holomorphic flows. We apply Theorem \ref{Kurspace} to the case that the foliation $\calf$ is also transversely K\"ahler with respect to the holomorphic structure. In this case El Kacimi Alaoui and Gmira have proven the following Stability Theorem (see also \cite{ElK88} for the equivalent  orbifold case):

\begin{theorem}[\cite{ElKGm97}]\label{transkahstab}
Let $\calf_0$ be a homologically oriented transversely holomorphic foliation on a compact manifold $M$ with a compatible transverse K\"ahler metric. Then there exists a neighborhood $U$ of the germ $\calf_0$ in the Kuranishi space $S$ such that for all $t\in U$ the holomorphic foliation $\calf_t$ is homologically oriented and has a compatible transverse K\"ahler metric $\gro_t$ depending smoothly on $t$.
\end{theorem} 

We apply this theorem to the case where the foliation $\calf$ has dimension one, that is to isometric transverse K\"ahler flows:

\begin{theorem}\label{transkahflowstab}
Let $(\calf_0,J,\gro^T)$ be an isometric transverse K\"ahler flow on a compact oriented manifold $M$. Then there exists a neighborhood $U$ of the germ $\calf_0$ in the Kuranishi space $S$ such that for all $t\in U$ the holomorphic flow $\calf_t$ has a compatible transverse K\"ahler metric $\gro^T_t$ making $(\calf_t,J_t,\gro^T_t)$ an isometric transverse K\"ahler flow. 
\end{theorem}

We are ready for

\begin{definition}\label{Sstabledef}
Let $(M,\cals_0)$ be a Sasaki manifold. We say that $\cals_0$ is {\bf $\cals$-stable} if there exists a neighborhood $N$ of $(\calf_0,\bar{J_0})$ in the Kuranishi space such that $(\calf_t,\bar{J_t})$ is Sasakian for all $t\in N$.
\end{definition}

Goertsches, Nozawa, and T\"oben proved that the basic Hodge numbers of a compact Sasaki manifold depend only on the underlying CR structure, Theorem 4.5 of \cite{GoNoTo12}, and more recently Ra\'zny \cite{Raz21} proved that the basic Hodge numbers of a compact Sasaki manifold are invariant under arbitrary deformations. The question arises as to whether the analogue of this holds for a general isometric transverse K\"ahler manifold. Generally, we do not know; however, we do have what we need, namely

\begin{lemma}\label{pq2lem}
There exists a neighborhood $N\subset S$ of the transverse K\"ahler flow $(\calf_0,J_0,\gro^T_0)$ in the Kuranishi space $S$ such that
$h^{p,q}(\calf_t,\bar{J_t})=h^{p,q}(\calf_0,\bar{J_0})$ for $p+q=2$ and for all $t\in N$. Furthermore, the holomorphic foliation $(\calf_t,\bar{J_t})$ has a compatible transverse K\"ahler form $\gro^T_t$.
\end{lemma}

\begin{proof}
We outline the proof following \cite{ElKGm97} which in turn followed \cite{KoSp60}.
We endow the spaces of smooth sections of vector bundles with the Fr\'echet topology. Then the space $\grO^{p,q}$ of smooth basic $(p,q)$-forms has a smooth family of transversely strongly elliptic essentially self adjoint 4th order differential operators
\begin{equation}\label{At}
A_t =\partial_t\bar{\partial_t}\bar{\partial_t}^*\partial_t^*+\bar{\partial_t}^*\partial_t^*\partial_t\bar{\partial_t} +                            \bar{\partial_t}^*\partial_t\partial_t^*\bar{\partial_t} +\bar{\partial_t}^*\bar{\partial_t} +\partial_t^*\partial_t.
\end{equation}
The kernel of $A_t$ denoted by $\bfF^{p,q}_t$ is given by  
\begin{equation}\label{kerAt}
\bfF^{p,q}_t=\{\gra\in\grO^{p,q}_B(\calf)~|~ \partial_t\gra=0, \quad \bar{\partial_t}\gra=0, \quad \bar{\partial_t}^*\partial_t^*\gra=0\}, 
\end{equation}
and we have the following orthogonal decomposition of smooth closed basic $(p,q)$ forms
\begin{equation}\label{Zpqdecomp}
Z^{p,q}_B(\calf)={\rm im}(\partial_t\bar{\partial_t})\oplus \bfF^{p,q}_t.
\end{equation}
So the cohomology groups $H^{p,q}_B(\calf)$ are represented by elements of $\bfF^{p,q}_t$. Thus, by Proposition 6.3 of \cite{ElKGm97} there is a neighborhood $N$ of the central fiber $(\calf_0,J_0,\gro^T_0)$ such that for all $t\in N$ the dimension of $\bfF^{1,1}_t$ equals $h^{1,1}_B(\calf_{\xi_t},\bar{J_t})$ and is independent of $t$, so  $h^{1,1}_B(\calf_{\xi_t},\bar{J_t})=h^{1,1}_B(\calf_0,\bar{J}_0)$ in $N$. But since the basic 2nd Betti number $b_2^B$ is independent of $t$ and we have
$$b^2_B=h^{2,0}(\calf_{\xi_t},\bar{J_t})+h^{1,1}_B(\calf_{\xi_t},\bar{J_t})+h^{0,2}_B(\calf_{\xi_t},\bar{J_t})=h^{1,1}(\calf_0,\bar{J_0})+2h^{2,0}_B(\calf_{\xi_t},\bar{J_t})$$
which implies that $h^{2,0}_B(\calf_{\xi_t},\bar{J_t})$ and $h^{0,2}_B(\calf_{\xi_t},\bar{J_t})=\overline{h^{2,0}_B(\calf_{\xi_t},\bar{J_t})}$ are also independent of $t$ for all $t\in N$ which proves the first result. The second result also follows by Theorem 6.4 of \cite{ElKGm97}.
\end{proof}

Since $h^{p,q}_B(\calf_\xi,J)$ are integer valued and $\ga^+(M,\calf_\xi)$ is path connected, Lemma \ref{pq2lem} implies

\begin{proposition}\label{h20CR}
Let $(M,\calf_\xi)$ be a compact transverse K\"ahler flow. Then $h^{p,q}_B(\calf_{\xi'})$ is independent of $\xi'\in \ga^+(M,\calf_\xi)$ for $p,q\leq 2$.
\end{proposition}

This proposition together with the fact that quasiregular Sasakian structures are dense in the Sasaki cone allows us to reduce our arguments to the quasiregular case. Given this we shall often use the well known correspondence between the transverse geometry of a quasiregular Sasakian structure and the projective algebraic geometry of its quotient orbifold \cite{BG05}.

\subsection{An Obstruction to $\cals$-stability}
We now consider obstructions to the stability of deformations of the transverse holomorphic foliations $(\calf_\xi,J)$. 

\begin{lemma}\label{varHodge2form}
Let $(\calf_0,\bar{J_0},d\eta_0)$ be the transverse K\"ahler flow of a Sasakian structure $\cals_0=(\xi_0,\eta_0,\Phi_0,g_0)$.
Under the deformation $(\calf_0,\bar{J_0},d\eta_0)\mapsto (\calf_t,\bar{J_t},d\eta)$, there is a neighborhood $N$ of $(\calf_0,\bar{J_0},d\eta_0)$ in the Kuranishi space such that for all $t\in N$ 
\begin{enumerate}
\item the $(2,0)$ component of $d\eta$ is $\partial$-closed with respect to $\bar{J_t}$, 
\item the $(0,2)$ component of $d\eta$ is $\bar{\partial}$-closed with respect to $\bar{J_t}$, 
\item the $(1,1)$ component of $d\eta$ is K\"ahler with respect to $\bar{J_t}$ if and only if $d\eta^{2,0}$ is holomorphic and  $d\eta^{0,2}$ is antiholomorphic,
\item  $d\eta^{2,0}\wedge d\eta^{0,2}+(d\eta^{1,1})^2>0$,
\item $h^{p,q}(\calf_t,\bar{J_t})=h^{p,q}(\calf_0,\bar{J_0})$ for $p+q=2$.
\end{enumerate}
\end{lemma}

\begin{proof}
The Hodge decomposition of the $d_B$-closed basic 2-form $d\eta$ with respect to the transverse holomorphic structure $(\calf_\xi,J_t)$ is given by Equation \eqref{basEuldecomp}. This shows that $d\eta^{2,0}$ is $\partial$-closed, $d\eta^{0,2}$ is $\bar{\partial}$-closed, and that 
$$\bar{\partial}d\eta^{2,0}+\partial d\eta^{1,1}=0, \qquad \bar{\partial}d\eta^{1,1}+\partial d\eta^{0,2}=0.$$ 
So $d\eta^{1,1}$ will be closed if and only if $d\eta^{2,0}$ is holomorphic and  $d\eta^{0,2}$ is antiholomorphic. Thus, in this case $d\eta^{1,1}\circ(\bar{J_t}\otimes \BOne)$ will be a transverse K\"ahler metric in a neighborhood of the central fiber $(\calf_\xi,\bar{J})$ which proves (1),(2), and (3). Item (4) follows from the Hodge decomposition and the fact that $\eta$ is a contact 1-form. Item (5) holds by Lemma \ref{pq2lem}.
\end{proof}

%we also have $h^{2,0}_B(t)=h^{2,0}_B$ and $h^{0,2}_B(t)=h^{0,2}_B$. So if we begin with a Sasakian structure $\gro^T=d\eta$ with $h^{2,0}_B=0$, we have $h^{2,0}_B(t)=0$ for all $t\in N$. This implies that $d\eta$ is K\"ahler with respect to $(\calf_0,\bar{J}_t)$.

%note that as in,  prove that there is a neighborhood $N_0$ of $(\calf_\xi,\bar{J},\gro^T)$ in the Kuranishi space such that $h^{1,1}_B(\calf_\xi,\bar{J_t})=h^{1,1}_B(\calf_\xi,\bar{J})$ for all $t\in N_0$. But then since the basic second Betti number

Applying Lemma \ref{varHodge2form} to Sasaki manifolds shows that if  $d\eta^{2,0}$  is a nonzero holomorphic section of $H^{2,0}(\calf_\xi,\bar{J})$, we can deform to a transverse K\"ahler structure which is not necessarily associated to a Sasakian structure since $\gro^T=d\eta^{1,1}\neq d\eta$. Indeed, in Theorem \ref{mainhyper} below we prove that this is the case for transverse hyperk\"ahler structures. For any transverse holomorphic deformation of a Sasakian structure, we view the holomorphic section $d\eta^{2,0}$ as an obstruction to $\cals$-stability\footnote{Nozawa identifies the $(0,2)$ component $(d\eta)^{0,2}$ as an obstruction to stability. Of course these are completely equivalent obstructions.}. 

%We consider a smooth family $\bar{J_t}$ of transversely K\"ahler contact structures of the form $\cals_t=(\xi_0,\eta_0,\Phi_t,g_t)$ with $\Phi_t=J_t+\xi_0\otimes\eta_0$ where the central fiber $(\xi_0,\eta_0,\Phi_0,g_0)$ is Sasakian. 

%Assuming that the basic Hodge numbers $h^{2,0}_B$ and $h^{0,2}_B$ vanish, one sees from Lemma \ref{varHodge2form} that these obstructions vanish in a neighborhood of the central fiber; hence, the compatibility conditions 
%\begin{equation}\label{compateqns}
%d\eta(J_t X,J_t Y)=d\eta(X,Y), \qquad g^T_t(J_tX,J_tY)=g^T_t(X,Y),
%\end{equation}
%hold for $t$ small enough where $X,Y$ are smooth sections of $\cald$. 

\subsection{Proof of Theorem \ref{no02thm} and Corollary \ref{CYcor}}
Given a Sasakian structure $\cals=(\xi,\eta,\Phi,g)$ there is an underlying transverse K\"ahler structure $(\calf,\bar{J},\gro^T)$ that satisfies $\gro^T=d\eta$ and is isometric with respect to the Reeb vector field $\xi$. However, given such an isometric transverse K\"ahler structure, the corresponding Sasakian structure is not unique. Clearly, the Sasakian structure $(\xi,\eta',\Phi',g')$ where $\eta'=\eta +\grz$ for $\grz$ a closed basic 1-form has the same underlying transverse K\"ahler structure $(\calf,\bar{J},d\eta)$. Here $\Phi'=(\BOne +\xi\otimes\grz)\circ \Phi$ and $g'=g+\eta\otimes\grz +\grz\otimes\eta +\grz\otimes\grz$. More generally, Corollary 1.7 in \cite{Noz14} shows that when $H^1(M,\bbr)=0$, the forgetful functor from the set of Sasakian structures on a closed manifold $M$ to the set of transversally K\"ahler flows is full. Explicitly, this corollary says that on a manifold with vanishing first Betti number, if the underlying transversally K\"ahler flows of two Sasakian structures are isomorphic, then the Sasakian structures are isomorphic. It remains to show that the vanishing of $h^{2,0}_B(\cals)$ and $h^{0,2}_B(\cals)$ implies that if the central fiber is a transversally K\"ahler flow $(\calf,\bar{J},d\eta)$ of a  Sasakian structure, there is a neighborhood of $(\calf,\bar{J},d\eta)$ in the Kuranishi space consisting of  transversally K\"ahler flows $(\calf_t,\bar{J_t},d\eta_t)$ of Sasakian structures $\cals_t$. By Lemma \ref{pq2lem} these Hodge numbers are independent of the Sasakian structure in $\ga^+(M,\calf_\xi)$. So Lemma \ref{varHodge2form} implies that $h^{2,0}_B(\cals_t)=0$ for $t\in N_0$, a small enough neighborhood of the central fiber. So when we deform the transverse holomorphic foliation $\calf,\bar{J}$, the basic Euler class $[d\eta]_B$ of $(\calf_t,\bar{J}_t)$ must remain type $(1,1)$. It then follows from the following Theorem 1.1 of \cite{Noz14} that the smooth family of flows $(\calf_t,\bar{J_t})$ are Sasakian in a possibly smaller neighborhood of $(\calf,\bar{J})$. 

\begin{theorem}[Nozawa, Theorem 1.1]
Let $(\calf_0,\bar{J_0},d\eta_0)$ be the underlying transversally K\"ahler flow of the Reeb vector field of a Sasakian structure $\cals_0$, and let $(\calf_t,\bar{J_t})$ be a smooth family of transversally holomorphic flows in a neighborhood $V$ of $(\calf_0,\bar{J_0},d\eta_0)$ in the Kuransishi space. If the basic Euler class is of degree $(1,1)$ for all $t\in V$, then there exists an open neighborhood $V_1$ of the central fiber in $V$ and a smooth family of Sasakian structures $\cals_t$ such that the underlying transversally holomorphic flow of the Reeb vector field of $\cals_t$ is $(\calf_t,\bar{J_t})$ for all $t\in V_1$.
\end{theorem}

This proves Theorem \ref{no02thm}.
The proof of Corollary \ref{CYcor} now follows as in Proposition 7.1.7 of \cite{Joy07}.

\section{Transverse K\"ahler Holonomy}
The irreducible transverse K\"ahler holonomy groups are 
$$U(n),\qquad SU(n),\qquad Sp(n)$$
which correspond to irreducible transverse K\"ahler geometry, transverse Calabi-Yau geometry, and transverse hyperk\"ahler geometry, respectively. Such general transverse structures were studied recently by Habib and Vezzoni \cite{HaVe15}. They can be defined as holomorphic foliations whose transverse holonomy group is contained in $SU(n)$. Here we are interested in their relation with Sasakian geometry, so we specialize to the case of a holomorphic foliation of dimension one, namely the characteristic Reeb foliation $\calf_\xi$.  Sasaki manifolds with transverse holonomy contained in $SU(n)$ are null-Sasaki having vanishing transverse Ricci curvature by the transverse Yau Theorem \cite{ElK,BGM06}. They are called contact Calabi-Yau manifolds in \cite{TomVe08}. 

Following Joyce \cite{Joy07,GHJ03} we deal with irreducible transverse Calabi-Yau and irreducible transverse hyperk\"ahler structures although we give the more general definitions below. Note that when $n=2$ we have the equality $SU(2)=Sp(1)$, so Calabi-Yau and hyperk\"ahler geometry coincide when $n=2$. We note that the condition of irreducibility is crucial for the following stability results.

\begin{remark}\label{sprem}
There is one other even dimensional irreducible Berger holonomy group that is related to Sasakian geometry, the group $Sp(n)\cdot Sp(1)$ whose transverse flows are twistor spaces of 3-Sasakian structures, cf. \cite{BG99}; however, generally they are not K\"ahler and therefore, are not treated in this paper. 
\end{remark}

\subsection{Transverse Irreducible Calabi-Yau Structures}
Since $c_1(\calf)=0$ the transverse geometry is the geometry of compact Calabi-Yau orbifolds which has been studied in \cite{Cam04} following the manifold case \cite{Bog78,Bea83}. 

\begin{definition}\label{transCYdef}
We say that a transverse K\"ahler flow $(\calf,\bar{J},\gro^T)$ on a compact manifold $M$ of dimension $2n+1$ is a {\bf transverse Calabi-Yau flow} if its transverse holonomy group is contained in $SU(n)$. This transverse Calabi-Yau structure is {\bf irreducible} if the transverse holonomy group equals $SU(n)$. We abbreviate irreducible transverse Calabi-Yau structures (flows) by {\bf ITCY}. The ITCY flow is said to be of {\bf Sasaki type} if $\gro^T=d\eta$ for some Sasakian structure $(\xi,\eta,\Phi,g)$.
\end{definition}

Calabi-Yau structures have holomorphic volume forms, so as expected transverse Calabi-Yau structures have transverse holomorphic volume forms, i.e holomorphic sections $\grO^T$ of $H^{n,0}_B$. Since as mentioned above Calabi-Yau structures coincide with hyperk\"ahler structures when $n=2$, we assume in this section that $n>2$.

We have following 

\begin{theorem}\label{thm1.2cor}
Let $M^{2n+1}$ be a compact manifold of dimension $2n+1$ with $b_1(M)=0$. If $M$ has an ITCY flow $(\calf,\bar{J},d\eta)$ of Sasaki type and $n>2$ then $(\calf,\bar{J},d\eta)$ is $\cals$-stable. Moreover, the Kuranishi space $S$ is a open set in $H^1(M,\Theta)$.
\end{theorem}

\begin{proof}
First we note that since transverse CY structures are null Sasakian, they are always quasiregular (\cite{BG05}, pg 246). So transverse CY structures are described by CY orbifolds. Moreover, since we consider irreducible Calabi-Yau orbifolds $X$, the transverse holonomy group is precisely $SU(n)$. Now since $n>2$, as noted on page 125 of \cite{Joy07}, the induced action of the holonomy group $SU(n)$ on $\grL^{p,0}(X)$ fixes no complex $(p,0)$ form for $0<p<n$, and this implies that the Hodge numbers $h^{p,0}$ vanish in this range. The first statement is then an immediate corollary of Theorem \ref{no02thm}. 

The second statement is an orbifold version of a result of Tian \cite{Tia86} which we now describe.  So we let $X$ be a compact K\"ahler orbifold and $\Gamma(X, \Omega^{p, q}(\Theta_X))$ be the set of global $(p, q)$-forms with coefficients in the sheaf of germs of holomorphic vector fields $\Theta_X$ or more generally for the sheaf of germs of any holomorphic tensor field. For the description of tensor fields on orbifolds we refer to \cite{BGK05} as well as Section 4.4.2 of \cite{BG05}. Generally, the canonical sheaf and the canonical orbisheaf are not equivalent; however, since transverse Calabi-Yau structures are null Sasakian structures, there are no branch divisors and Tian's proof is straightforward to generalize. From the GHS Theorem \ref{Kurspace} we need to prove the existence of a one parameter family of  solutions $\omega(t) \in \Gamma(X, \Omega^{0, 1}(TX))$ with
\begin{equation}
\label{deformation_equation_1}
\bar{\partial} \omega(t) + \frac{1}{2} [\omega(t), \omega(t)] = 0, ~ \omega(0) = 0
\end{equation}
give a deformation of complex structures over $X$. 

By the Taylor expansion (at the singular points, we consider the corresponding local covering spaces), we have $\omega(t)  = w_1 t + w_2 t^2 + \cdots$ which we plug into \eqref{deformation_equation_1}. Given $\omega_1 \in \Gamma(X, \Omega^{0, 1}(TX))$, we then need to solve the following system of equations inductively
\begin{equation}
\label{deformation_equation_2}
\bar{\partial} \omega_N + \frac{1}{2} \sum_{i=1}^{N-1} [\omega_i, \omega_{N-i}] = 0,  ~ (N \geq 2).	
\end{equation}

Now we want to change \eqref{deformation_equation_2} a bit. Since the canonical orbisheaf $K_X$ is  trivial, we have a natural isomorphism
\begin{align*}
i_q : \Gamma(X, \Omega^{0, q}(T_X)) &\to \Gamma(X, \Omega^{n-1, q}). \\
\end{align*}
For every $\Omega^{0, q}(T_X)$, locally we have
$$
\phi = \sum_{\substack{i, J \\ |J|=q} } f^i_{\bar{J}} \frac{\partial}{\partial z^i}\otimes d \bar{z}^J,
$$
and
$$
i_q(\phi) = dz^1 \wedge \cdots \wedge dz^n(\phi).
$$
It is easy to check that $i_q$ is well-defined and isomorphic. Our goal is to replace $\omega_i$ in \eqref{deformation_equation_2}. To do that, we define
\begin{align*}
[i_1(\omega_1), i_1(\omega_2)] := i_2 [\omega_1, \omega_2 ].
\end{align*}

Thus given $\omega_1 \in \Gamma(X, \Omega^{n-1,1})$, we need to solve the following system of equations inductively
\begin{equation}\label{defeqn3}
\bar{\partial} \omega_N + \frac{1}{2} \sum_{i=1}^{N-1} [\omega_i, \omega_{N-i}] = 0,  \qquad  \text{for $N \geq 2$},	
\end{equation}
where $\omega_i \in \Gamma(X, \Omega^{n-1, 1}), ~ i = 2, 3, \ldots, N-1$.
Since the proof of Lemma 3.1 of \cite{Tia86} is local, it also holds in the orbifold case on the local uniformizing neighborhoods.

\begin{lemma}
Let $\omega_1, \omega_2 \in \Gamma(X, \Omega^{n-1, 1})$, then
$$
[\omega_1, \omega_2] = \partial (i^{-1}(\omega_1) \lrcorner \omega_2) - \#(\partial \omega_1) \wedge \omega_2 + \omega_1 \wedge \#(\partial \omega_2).
$$
\end{lemma}

Then since the $\partial \bar{\partial}$-lemma for orbifold Hodge theory follows from the transverse version in \cite{ElK}, the remainder of Tian's argument applies to our case. Indeed the proof of Theorem 1 of \cite{Tia86} goes through verbatim.
\end{proof}

%\begin{theorem}
%If $X$ is a compact K\"ahler orbifold with $c_1^{orb}(X) = 0$, then the local universal deformation space of $X$ is isomorphic to an open set in $H^1(X, \Theta_X)$.
%\end{theorem}

\subsection{Transverse Irreducible Hyperk\"ahler Structures}
The seminal work on hyperk\"ahler manifolds is \cite{HKLR}.
Hyperk\"ahler structures are a particular type of quaternionic structure to which we refer to Chapter 12 of \cite{BG05}, Chapter 10 of \cite{Joy07}, and Chapter 3 of \cite{GHJ03} as well as the standard references \cite{Huy99,Ver05}.

Although we give the more general definition, henceforth by hyperk\"ahler we shall mean the irreducible case, ${\rm Hol}=Sp(n)$. We abbreviate irreducible transverse hyperk\"ahler structures by {\bf ITHK}.

\begin{definition}\label{transhyperkahdef}
We say that a transverse K\"ahler flow $(\calf,\bar{J},\gro^T)$ on a compact manifold $M$ of dimension $4n+1$ is a {\bf transverse hyperk\"ahler flow} if its transverse holonomy group is contained in $Sp(n)$. The transverse hyperk\"ahler structure is {\bf irreducible} if the transverse holonomy group equals $Sp(n)$.
\end{definition}

\begin{remark}\label{altdef}
An equivalent definition of transverse hyperk\"ahler is that the contact bundle $\cald$ admits three almost complex structures $\{I_i\}_{i=1}^3$ that satisfy the algebra of the quaternions 
\begin{equation}\label{quatalg}
I_iI_j=-\grd_{ij}\BOne +\gre_{ijk}I_k,
\end{equation}
and the induced transverse antisymmetric forms $\gro_i^T=g\circ (I_i\otimes \BOne)$ are covariantly constant ($\nabla^T\gro^T_i=0$) with respect to the transverse Levi-Civita connection $\nabla^T$. 
\end{remark}

It immediately follows from the definition that $M$ has real dimension $4n+1$. Since we have an inclusion of holonomy groups $Sp(n)\subset SU(2n)$ a transverse hyperk\"ahler structure is automatically a transverse null K\"ahler structure. We want to know when this transverse K\"ahler structure is Sasakian. There is a 1-1 correspondence between transverse hyperk\"ahler flows and hyperk\"ahler orbifolds, and these orbifolds are projective algebraic if and only if  the canonical bundle $c_1(K_X)\in H^2_{orb}(X,\bbz)$.

%So the transverse hyperk\"ahler structure can be represented by a unique polarized hyperk\"ahler orbifold 

%$(X,\{\gro_i)\}_{i=1}^3$ 

\begin{lemma}\label{hyperkahsaslem}
Let $(\xi,\eta,\Phi,g)$ be a contact metric structure with a transverse hyperk\"ahler structure $\{I_i\}_{i=1}^3$. Then fixing a transverse complex structure, say $I_1$, gives a null transverse K\"ahler structure $\gro_1^T$. It is a Sasakian structure $\cals_1=(\xi,\eta,\Phi_1,g)$ with $\Phi_1=I_1\oplus \xi\otimes \eta$ if and only if $\gro^T_1=d\eta$. 
In this case we say that the hyperk\"ahler structure is {\it associated} to the (necessarily null) Sasakian structure $\cals_1$, or conversely the null Sasakian structure  $\cals_1$ is associated to the hyperk\"ahler structure $\{I_i\}_{i=1}^3$.
\end{lemma}

\begin{proof}
The condition $\nabla^T\gro_i=0$ implies that $\gro_i$ are closed. But as in Lemma 2.2 of \cite{Hit87} this implies that the transverse almost complex structures $I_i$ are integrable and that the forms $\gro_i^T$ are K\"ahler with respect to the complex structure $I_i$. But clearly $\pounds_\xi\Phi_i=0$ for $i=1,2,3$ since $I_i$ are endomorphisms of $\cald$ and $\pounds_\xi\eta=0$. This implies $\pounds_\xi\gro_i^T=0$. Moreover, it is null, that is, $c_1(\calf_\xi)=0$ which implies $c_1(X)_\bbr=0$ which in turn implies that $c_1(\cald)$ is a torsion class. 
\end{proof}

\begin{remark}\label{transcomsymp}
As in the manifold case a transverse hyperk\"ahler structure defines a transverse K\"ahler structure with a {\bf transverse complex symplectic structure}. Explicitly, if $(I_1,\gro^T_1)$ defines the underlying transverse K\"ahler structure of the transverse hyperk\"ahler structure, the complex 2-form $\gro^T_2+i\gro^T_3$ satifies $\gro_+^n\neq 0$ everwhere, and thus defines a transverse complex symplectic structure. Conversely, if we have a transverse K\"ahler structure $(J,\gro^T)$ together with a transverse holomorphic symplectic 2-form $\gro_\bbc^T$ that is covariantly constant with respect to the transverse Levi-Civita connnection $\nabla^T$ the conditions $\nabla^T\gro^T=\nabla^TJ=\nabla^T\gro^T_\bbc= 0$ forces the transverse holonomy to lie in $Sp(n)$ where the real codimension of the foliation is $4n$ as in \cite{Joy07} Section 10.4. This gives an equivalence between transverse hyperk\"ahler structures and tranverse K\"ahler structures with a transverse complex symplectic structure.
\end{remark}

Given a transverse hyperk\"ahler structure we fix a transverse K\"ahler structure $(I_1,\gro^T_1)$ and its transverse complex symplectic structure $\gro_+=\gro_2+i\gro_3$. We now consider the proof of Theorem \ref{transhyperkahthm}. First, we note that the 2nd statement in the theorem follows from \cite{Cam04}. So it suffices to prove

\begin{theorem}\label{mainhyper}
Let $\cals_1$ be an (ITHK) Sasaki structure on a compact manifold $M$ with $b_1(M)=0$ with the transverse K\"ahler structure defined by  $I_1\in \{I_i\}_{i=1}^3.$ Then there are transverse K\"ahler deformations in the Kuranishi space $(S,0)^T$ that are not Sasakian. So $\cals_1$ is $\cals$-unstable.
\end{theorem}

\begin{proof}
Any null Sasakian structure $\cals$ is quasiregular, so the quotient  by the $S^1$ action generated by the Reeb vector field $\xi$ is a K\"ahler polarized hyperk\"ahler orbifold $(X_1,\gro_1)$ which represents the transverse hyperk\"ahler structure. Letting the Sasakian structure $\cals_1$ be the central fiber in the Kuranishi space $S$ of transverse holomorphic deformations, there is a neighborhood $N\subset S$ of $X_1$ such that all $X_t\in N$ are transversely K\"ahler by Theorem \ref{transkahflowstab}. Now the transverse hyperk\"ahler structure gives a 2-sphere's worth of transverse complex structures $I_t$ defined by 
\begin{equation}\label{2spheqn}
I_t=t_1I_1+t_2I_2+t_3I_3, \quad \text{with $t=(t_1,t_2,t_3)$ and $t_1^2+t_2^2+t_3^2=1$}.
\end{equation}
We also have a 2-sphere's worth of transverse K\"ahler forms 
\begin{equation}\label{2sphcomplex}
\gro^T_t=\sum_{i=1}^3t_i\gro^T_i=\sum_{i=1}^3t_ig\circ (I_i\otimes \BOne)=g\circ (I_t \otimes \BOne), \qquad \sum_{i=1}^3t^2_i=1.
\end{equation}
Since the transverse geometry is that of a K\"ahler orbifold, this gives rise to the twistor space $\calt(X)$, a complex orbifold which is diffeomorphic as orbifolds to the product $X\times \bbc\bbp^1$, but whose complex structure is not the product structure.  It is more convenient to use stereographic coordinates $z\in\bbc$ defined by
\begin{equation}\label{2spheqn2}
t=(t_1,t_2,t_3)= \Bigl(\frac{1-|z|^2}{1+|z|^2},-\frac{z+\bar{z}}{1+|z|^2},i\frac{z-\bar{z}}{1+|z|^2}\Bigr)
\end{equation}
with corresponding complex structure $I_z$. For each $z\in\bbc\bbp^1$, there is an associated transverse K\"ahler structure. If we begin with a Sasakian structure with respect to $\Phi_0=I_1+\xi\otimes\eta$ and consider deformations of the transverse holomorphic structure leaving the transverse hyperk\"ahler structure invariant, we obtain the complex structures $I_z$ for $z\in\bbc\bbp^1$. From this we get an induced complex structure on $\calt(X)$ as follows. Using the natural projection $p:\calt(X)\lra \bbc\bbp^1$ we can lift the standard complex structure $I_0$ on $\bbc\bbp^1$ to $\calt(X)$ and denote it by $p^*I_0$, and define the complex structure on $\calt(X)$ by $J=I_z+p^*I_0$. Of course, this makes the map $p:\calt(X)\lra \bbc\bbp^1$ holomorphic.
Furthermore, we have a double fibration, a la Penrose\footnote{This arises from Penrose's nonlinear graviton \cite{Pen76} and is amply treated in books \cite{Wel82,WaWe90,MaWo96}.}, (cf. Diagram 12.6.4 of \cite{BG05})
\begin{equation}\label{twistdoubfibr}
\begin{matrix} &&\calt(X)&& \\
                  &\fract{p}{\swarrow} && \searrow & \\
                   \bbc\bbp^1 &&\leadsto && X
\end{matrix}
\end{equation}
which gives a correspondence: points $z\in \bbc\bbp^1\simeq S^2$ correspond to complex structures
$I_z$ on $X$ in the given hyperk\"ahler structure
$\boldsymbol{\cali};$ points $x\in X$ correspond to rational
curves in $\calt(X)$ with normal bundle $2n\calo(1),$ called {\it twistor lines}. The general point is that the holomorphic data on the twistor space $\calt(X)$ encodes the hyperk\"ahler data on $X$. We note that generally the twistor space $(\calt(X),J)$ is not K\"ahler.

Let $(\xi,\eta,\Phi_0,g)$ be a Sasakian structure which is associated to the transverse hyperk\"ahler structure $\{I_i\}_{i=1}^3$. Then by Lemma \ref{hyperkahsaslem} $\gro_1^T$ is a transverse K\"ahler form satisfying $\gro_1^T=d\eta$. Moreover, $\gro_+=\gro_2+i\gro_3$ is a transverse holomorphic section of $H^{2,0}(\calf_\xi)$ and $\gro_-=\gro_2-i\gro_3$ is transverse anti-holomorphic section of $H^{0,2}(\calf_\xi)$. So $h^{2,0}\neq 0\neq h^{0,2}$. Since these transverse hyperk\"ahler structures are null Sasakian, the transverse geometry is that of hyperk\"ahler orbifolds $X$ with cyclic isotropy groups. Now let us deform the complex structure of the orbifold $X$ in a disc in $S^2$ centered around $I_1$. This gives the twisted $(2,0)$ form 
\begin{equation}\label{tw2form}
\gro_z=\gro_++2z\gro_1-z^2\gro_-
\end{equation}
as a section of $p^*\calo(2)\otimes \grO^{2}(\calt(X))$ and representing the variation of Hodge structures on $X$. Thus, $\gro_z$ has two interpretations: (1) as a holomorphic $(2,0)$-form on the twistor space, and (2) as a holomorphic $(2,0)$-form on each member $(X,I_z)$ of the family of complex orbifolds parameterized by a holomorphic section of $\calo(2)$ on $\bbc\bbp^1$. Now $\gro_z$ defines a class in $H^2(X,\bbq)$ for at most a countable number of $z\in\bbc\bbp^1$. Thus, for only a countable number of points $z\in \bbc\bbp^1$ will $[\gro_z]$ lie in an integral Hodge lattice $\grL=H^2_{orb}(X,\bbz)$. That is, for at most a countable number of $z\in\bbc\bbp^1$ we have $[\gro_z]\in H^2_{orb}(X,\bbz)\cap H^{2,0}(X_z)$. Then since for compact irreducible  hyperk\"ahler orbifolds $h^{2,0}=h^{0,2}=1$ \cite{Fuj83}, it follows that in the case when $[\gro_z]\not\in H^2_{orb}(X,\bbz)$ the transverse complex structure admits no (non-zero) integer lattice in $H^{1,1}_B(\calf_\xi)$. So transversally, there is an integral lattice and a transverse Picard group ${\rm Pic}^T(M,\calf_\xi)$ of isomorphism classes of holomorphic orbi-line-bundles over $(X,I_z)$ for at most a countable number of $z\in\bbc\bbp^1$. Thus, in such cases for all but a countable number of points $z\in\bbc\bbp^1$, we have ${\rm Pic}^T(M,\calf_\xi)={\rm Pic}^{orb}(X,I_z)=0$, and so all but a countable number of points $z\in \bbc\bbp^1$ are non-algebraic hyperk\"ahler orbifolds. These cannot represent the transverse K\"ahler structure of a Sasakian structure, since null Sasakian structures are algebraic, that is they are the total space of an orbibundle over a projective algebraic orbifold \cite{BG05}.
\end{proof}

%The existence of a lattice $H^{1,1}(X_z,\bbz)$ would imply a basis 
%$$\{a_1,a_2,...,a_k, [\omega_z], [\bar{\omega}_z]\}$$ 
%of $H^2(X,\bbr)$ where $a_i$ are all integer classes (of type $(1,1))$. Then since we know that we always have
%a basis $\{b_1, b_2, ....,b_k, b_{k+1}, b_{k+2}\}$ of $H^2(X,\bbr)$ where $b_i$ are all integer classes (of potentially scrambled type), a dimension count would
%force $[\omega_z]$ and $[\bar{\omega_z}]$ to be integer classes as well. Hence if we know they are not integer classes (which happens for most values of $z$), we get a contradiction to the existence of $a_1, ....,a_k$ and hence the lattice $H^{1,1}(X_z,\bbz)$. 

\subsection{Trivial Restricted Transverse Holonomy}\label{trivialhol}
Finally, we briefly consider the case when the restricted transverse holonomy group $\Hol^0(\calf)$ is the identity. For simplicity we only consider the regular case of $S^1$ bundles over a polarized Abelian variety. In this case the restricted transverse holonomy group is the identity. They are nilmanifolds $\gN_\bfl$ of dimension $2n+1$ formed as quotients of the $(2n+1)$-dimensional Heisenberg group $\gH(\bbr)$ by a lattice subgroup $\grG_\bfl$ where $\bfl=(l_1,\ldots,l_n)$ is a $\bbz$-vector whose components are positive and satisfy the divisibility conditions $l_j|l_{j+1}$ for $j=1,\dots,n-1$ \cite{Fol04}. Now $\gN_\bfl$ has a canonical strictly pseudoconvex CR structure $(\cald,J)$. In fact it has a compatible Sasakian structure $\cals_\bfl$, unique up to equivalence, and $\cals_\bfl$ has constant $\Phi$-sectional curvature $-3$  \cite{Boy09}.  These nilmanifolds are both homogeneous and Sasakian, but they are not Sasaki homogeneous. Moreover, Folland shows that there is a 1-1 correspondence between equivalence classes of such CR structures and polarized Abelian varieties $(\bbc^n/\grL_\bfl,L)$ equipped with a positive line bundle $L$ where the lattice $\grL_\bfl$ is the image of $\grG_\bfl$ under the natural projection $\pi:\gN_\bfl\ra{2.5} \bbc^n/\grL_\bfl$. Here $l_1$ is the largest positive integer such that $c_1(L)/l_1$ is primitive in $H^2(\bbc^n/\grL_\bfl,\bbz)$. The first homology group of such nilmanifolds is $H_1(\gN_\bfl,\bbz)=\bbz^{2n}+\bbz_{l_1}$.  Nozawa \cite{Noz14} proved that all such $(\gN_\bfl,\cals_\bfl)$ are $\cals$-unstable when $n \geq 2$.

\subsection{Reducible Transverse K\"ahler Holonomy}\label{redholosect}
Here we consider the case of reducible K\"ahler holonomy, namely,  the {\it join} of two quasiregular Sasaki manifolds $M_1,M_2$ defined in \cite{BG00a,BGO06} and developed further in \cite{BHLT16}. Recall that for any pair of relatively prime positive integers $(l_1,l_2)=\bfl$ we define the join $M\star_{\bfl} M_2$ of two quasiregular Sasaki manifolds $M_1,M_2$ with Reeb vector fields $\xi_1,\xi_2$ respectively, by the quotient of $M_1\times M_2$ by the $S^1$ generated by the vector 
\begin{equation}\label{joinvf}
L_\bfl=\frac{1}{2l_1}\xi_1-\frac{1}{2l_2}\xi_2
\end{equation}
where $\xi_i$ is the Reeb field of the Sasakian structure on $M_i$. This gives rise to the commutative diagram
\begin{equation}\label{joindia}
\begin{matrix}  M_1\times M_2 &&& \\
                          &\searrow\pi_L && \\
                          \decdnar{\pi_{2}} && M_1\star_\bfl M_2 &\\
                          &\swarrow\pi_1 && \\
                         N_1\times N_2&&& 
\end{matrix}
\end{equation}
where $N_i$ are the quotient orbifolds of $M_i$, and the Reeb vector field of the induced Sasakian structure on $M_1\star_\bfl M_2$ is given by
\begin{equation}\label{Reebjoin}
\xi_\bfl=\frac{1}{2l_1}\xi_1+\frac{1}{2l_2}\xi_2.
\end{equation}
We now have

\begin{corollary}\label{stablejoin}
Let $\calm_\bfl=M_1\star_\bfl M_2$ be the join of quasiregular Sasaki manifolds $M_i$, $i=1,2$. Suppose also that $b_1(M_i)=0$ for $i=1,2$, and the basic Hodge numbers $h^{0,2}_B(M_i)$ also vanish. Then every Sasakian structure in the Sasaki cone $\gt^+_\bfl$ of $\calm_\bfl$ is $\cals$-stable. 
\end{corollary}

\begin{proof}
Since by \cite{GoNoTo12} the basic Hodge numbers depend only on the underlying CR structure, it suffices to prove the corollary for the Reeb field \eqref{Reebjoin} which is quasiregular. This amounts to computing the Hodge numbers of the product orbifold $N_1\times N_2$. By the Hodge-Kunneth formula (\cite{Voi02} page 286) we have
$$H^{0,2}(N_1\times N_2)=H^{0,2}(N_1)\otimes H^{0,0}(N_2) + H^{0,0}(N_1)\otimes H^{0,2}(N_2) +H^{0,1}(N_1)\otimes H^{0,1}(N_2).$$
This implies that 
$$h^{0,2}_B(\calm_\bfl)=h^{0,2}_B(M_1)+h^{0,2}_B(M_2) +h^{0,1}_B(M_1)h^{0,1}_B(M_2)$$
which vanishes by hypothesis and the injectivity of $H^1_B\lra H^1(M)$. The result then follows from Theorem \ref{no02thm}.
\end{proof}

We remark that the hypothesis of the corollary implies, using Theorem  \ref{no02thm} that the Sasakian structures on $M_i$ are both $\cals$-stable. However, we do not know whether generally the join of $\cals$-stable Sasakian structures is $\cals$-stable.

\subsection{Fiber Joins and $\cals$-Stability}
There is another type of join construction due to Yamazaki \cite{Yam99} which describes a construction of K-contact structures on sphere bundles over a symplectic manifold. Given a compact symplectic manifold $N$ with $d+1$ integral symplectic forms $\gro_j$, not necessarily distinct. Let $L_j$ be the complex line bundle on $N$ such that $c_1(L_j)=[\gro_j]$, then Yamazaki shows that the unit sphere bundle in the complex vector bundle $\oplus_{j=1}^{d+1}L_j^*$ has a natural K-contact structure associated to each Reeb vector field in the Sasaki cone $\gt^+_{sph}$ of the sphere $S^{2d+1}$. The manifold is denoted by $M=M_1\star_f\cdots\star_f M_{d+1}$ where $M_j$ is principal $S^1$ bundle associated to $L_j$. Moreover, it is easy to see that this K-contact structure is Sasakian if $N$ is a projective variety and $\gro_j$ are integral K\"ahler forms \cite{BoTo20b}. It was also shown there that such Sasakian structures come in two types, cone decomposable fiber joins and cone indecomposable fiber joins. The former is equivalent to a special case of the joins described in Section \ref{redholosect}; however, it follows from Proposition 3.8 (2) of \cite{BoTo20b} that the cone indecomposable fiber joins have irreducible $U(n)$ transverse holonomy. Nevertheless, in either case we have

\begin{corollary}\label{fiberjoincor}
Let $N$ be a smooth projective algebraic variety with $b_1(N)=0$ and integral K\"ahler forms $\gro_j$ and let $M=M_1\star_f\cdots\star_f M_{d+1}$ be a fiber join with its spherical Sasaki subcone $\gt^+_{sph}$ of Sasakian structures on $M$. Assume also that the Hodge number $h^{0,2}(N)$ vanishes. Then every Sasakian structure in $\gt^+_{sph}$ is $\cals$-stable.
\end{corollary}

\begin{proof}
Again since the basic Hodge numbers depend only on the underlying CR structure \cite{GoNoTo12}, we can choose the regular Reeb vector field $\xi$ in $\gt^+_{sph}$. 
In this case the transverse holomorphic structure is isomorphic to the complex structure of the quotient manifold, namely the projectivization $\bbp(\oplus_{j=1}^{d+1}L_j^*)$ which is the total space of a $\bbc\bbp^d$-bundle over $N$ (\cite{BoTo20b}, Section 3.3). The cohomology ring of such projective bundles is well known (\cite{BoTu82}, page 270) to be $H^*(\bbp(\oplus_{j=1}^{d+1}L_j^*))=$
$$H^*(N)[x]/\bigl(x^{d+1}+e_1(c_1(L_1),\ldots,c_1(L_{d+1}))x^d+\cdots +e_{d+1}(c_1(L_1),\ldots,c_1(L_{d+1}))\bigr)$$
where $e_i$ denotes the ith elementary symmetric function, and $x$ is a global generator of the cohomology of $\bbp(\oplus_{j=1}^{d+1}L_j^*)$ which when restricted to each fiber generates the cohomology of $\bbc\bbp^d$. Now the class $x$ is represented by a $(1,1)$ form on $\bbp(\oplus_{j=1}^{d+1}L_j^*)$. So the only nonvanishing element of $H^{0,2}(\bbp(\oplus_{j=1}^{d+1}L_j^*))$ can come from an element of $H^{0,2}(N)$. But this clearly implies $h^{0,2}(\bbp(\oplus_{j=1}^{d+1}L_j^*))=h^{0,2}(N)$, so the corollary follows from Theorem \ref{no02thm}.
\end{proof}

%It is easily seen from Diagram \eqref{joindia} that
%\begin{lemma}\label{basHodjoin}
%The basic Hodge numbers of $M_1\star_\bfl M_2$ are the sum of the basic Hodge numbers of 
%\end{lemma}

\def\cprime{$'$} \def\cprime{$'$} \def\cprime{$'$} \def\cprime{$'$}
  \def\cprime{$'$} \def\cprime{$'$} \def\cprime{$'$} \def\cprime{$'$}
  \def\cdprime{$''$} \def\cprime{$'$} \def\cprime{$'$} \def\cprime{$'$}
  \def\cprime{$'$}
\providecommand{\bysame}{\leavevmode\hbox to3em{\hrulefill}\thinspace}
\providecommand{\MR}{\relax\ifhmode\unskip\space\fi MR }
\providecommand{\MRhref}[2]{%
  \href{http://www.ams.org/mathscinet-getitem?mr=#1}{#2}
}
\providecommand{\href}[2]{#2}

\end{document}